\documentclass[10pt]{article}
\usepackage[margin=0.95in]{geometry}
\usepackage{amsmath,amsthm,amssymb,graphicx}
\usepackage{tikz}
\usepackage{pgfplots}
\pgfplotsset{compat=1.18}

\usepackage[dvipsnames]{xcolor}
\usepackage{hyperref}
\hypersetup{colorlinks=true, citecolor=black, linkcolor=black, urlcolor=black, pdfstartview={FitH}}

\newtheorem{theorem}{Theorem}
\newtheorem{lemma}[theorem]{Lemma}

\theoremstyle{remark}
\newtheorem*{remark}{Remark}

\DeclareMathOperator{\area}{area}
\DeclareMathOperator{\conv}{conv}
\DeclareMathOperator{\diam}{diam}

\title{A reduced planar body with area greater than $\pi\Delta^2/4$}
\author{Scott Duke Kominers%
\thanks{%
\textsc{Harvard Business School; Department of Economics and Center of
Mathematical Sciences and Applications, Harvard University; and a16z crypto.}
I used LLMs to assist with computations and coding in the preparation of this article, especially GPT-5.5~Pro and Claude~4.8 Opus (both accessed in part via Poe with the support of Quora, where I am an advisor); in particular, the reduced body $R$ was found by GPT-5.5~Pro via semi-directed search. I especially appreciate a thorough review from Refine.ink. The problem, methods, and eventual written form are my own; and of course any errors remain my responsibility. This work was conducted while I was visiting the Technological Innovation, Entrepreneurship, and Strategic Management (TIES) Group at the MIT Sloan School of Management; I greatly appreciate their hospitality.\smallskip%
\newline\indent%
\emph{2020 Mathematics Subject Classification.} Primary 52A40; Secondary 52A10, 52A38.
\newline\indent%
\emph{Key words and phrases.} Reduced convex body, thickness, minimal width, support function, constant width.
}}

\date{June 26, 2026}

\begin{document}
\maketitle
\vspace{-2.5em}

\begin{abstract}
We construct a reduced planar convex body $R$ with thickness $\Delta(R)=1$ and
\[
 \area(R)=0.786215\ldots>0.785398\ldots=\frac{\pi}{4}.
\]
Thus $R$ is a counterexample to Lassak's conjectured upper bound
$\area\le (\pi/4)\Delta^2$ for planar reduced bodies. The construction is given by an explicit support function, and the proofs use only elementary support-function, width, area, and contact-point computations.
\end{abstract}

\section{Introduction}
A \emph{convex body} is a compact convex subset of $\mathbb R^2$ with non-empty interior. The \emph{support function} of a convex body $C$ is
\[
 h_C(\theta)=\max_{x\in C} \{x\cdot u_\theta\},
 \qquad u_\theta=(\cos\theta,\sin\theta);
\]
the \emph{width of $C$ in direction $u_\theta$} is
\[
 w_C(\theta)=h_C(\theta)+h_C(\theta+\pi);
\]
and $C$'s \emph{thickness} is $\Delta(C)=\min_\theta w_C(\theta)$. A convex body $R$ is said to be \emph{reduced} \cite{Heil1978} if every convex body $K\subsetneq R$ satisfies $\Delta(K)<\Delta(R)$.

P\'al's isominwidth theorem \cite{Pal1921} determines the least area forced by a given thickness, attained by the equilateral triangle. The opposite question --- the greatest area --- is unbounded without a restriction because a convex body of thickness $\Delta$ can enclose arbitrarily large area (as a long, thin rectangle shows). Thus, for the area-maximization question we consider reduced bodies, which are exactly the convex bodies that are minimal for their thickness; since a reduced body of thickness $\Delta$ has diameter at most $\sqrt2\,\Delta$ \cite{Lassak1990}, the greatest area among them is finite. 

Lassak conjectured \cite{Lassak2005} that every planar reduced body $R$ satisfies
\begin{equation}\label{eq:conj}
 \area(R)\le \frac{\pi}{4}\,\Delta(R)^2.
\end{equation}
The bodies conjectured to be extremal are the disk of width $\Delta(R)$ and the quarter-disk of radius $\Delta(R)$, both of which achieve equality in~\eqref{eq:conj}; see \cite[Problem 3]{LassakMartini2011}.

The conjecture~\eqref{eq:conj} holds in its two best-understood cases. For reduced polygons --- regular odd-gons are examples, and the class is surveyed in \cite{LassakMartini2011} --- Lassak proved \eqref{eq:conj} with strict inequality, the value $\pi/4$ being approached only as the number of sides grows and the polygons tend to the disk \cite{Lassak2005}. For bodies of constant width, \eqref{eq:conj} follows from Barbier's theorem and the isoperimetric inequality. A counterexample can therefore be neither a polygon nor of constant width. Thus by the planar structure theory (see Section~\ref{sec:structure}) any counterexample to~\eqref{eq:conj} must combine genuine arcs with straight segments. In this note, we construct such a body, disproving the conjecture.

We exhibit an explicit convex body for which \eqref{eq:conj} fails, based on a deformation of the quarter-disk. In the quarter-disk, the constant-width structure that underlies every reduced body is degenerate: of the two opposite arcs that should together form a curve of constant width $\Delta$, one has collapsed to a single point, i.e., the center of the quarter-circle. Restoring that second curve --- ``blunting'' the quarter-disk by spreading that concentrated curvature onto a genuine arc --- yields a body bounded by two true constant-width arcs and two short segments, whose active arc now sweeps just past a quarter-turn, $\varphi=2m>\pi/2$. The resulting body $R$ is convex, reduced, and of thickness $1$, yet it encloses more area than the disk of width $1$ --- it has $\area(R)=0.786215\ldots>\pi/4$, thus violating \eqref{eq:conj}. The construction and analysis are fully elementary: we give an explicit support function, and establish the properties and area of $R$ via direct computation.

Beyond refuting \eqref{eq:conj}, the example reopens the underlying extremal problem. The maximal area of a planar reduced body of thickness $1$ --- long expected to equal that of the disk, namely $\pi/4$ --- is now known to exceed $\pi/4$, yet its exact value, and whether it can be attained, are open. We make no claim that $R$ is optimal, so $\area(R)$ is only a lower bound for the maximal area.

\section{The construction}\label{sec:construction}
We put
\begin{gather*}
 m=\frac{15}{16},\qquad \varphi=2m=\frac{15}{8},\qquad
 \psi=\pi-2m=\pi-\varphi,\\
 T=\cot m=\tan\frac{\psi}{2},\qquad
 k=\frac12\cot^2 m,
 \qquad \varepsilon=\frac{13}{20},
\end{gather*}
and
\[
 q=\frac{2\sin m}{m+\sin m\cos m}.
\]
Numerically, we have
\[
 T=0.7341755922917413\ldots,\quad
 k=0.2695069001584646\ldots,\quad
 q=1.1397054320523979\ldots.
\]
For $0\le \theta\le \varphi$, we consider
\[
 r(\theta)=k+\varepsilon\bigl(1-q\cos(\theta-m)\bigr),
\]
a constant $k$ plus the \emph{perturbation} $\varepsilon\bigl(1-q\cos(\theta-m)\bigr)=r(\theta)-k$, and set
\[
 v_\theta=(-\sin\theta,\cos\theta),\qquad
 c(\theta)=\int_0^\theta r(t)v_t\,dt.
\]
The value of $q$ is chosen so that $c(\varphi)$ takes the same value as it would for the constant function $r\equiv k$; since $c(\varphi)=\int_0^{2m}r(\theta)v_\theta\,d\theta$, this requires the perturbation to integrate to zero against $v_\theta$, i.e.:
\begin{equation}
 \int_0^{2m}\bigl(r(\theta)-k\bigr)v_\theta\,d\theta=0.
\label{eq:perturbation-integral}
\end{equation}
Identifying $\mathbb R^2$ with $\mathbb C$, so that $v_\theta=ie^{i\theta}$, and substituting $x=\theta-m$, the left-hand side of \eqref{eq:perturbation-integral} becomes 
\[
 \varepsilon\,ie^{im}\!\int_{-m}^{m}(1-q\cos x)e^{ix}\,dx;
\]
the imaginary part of this integral vanishes by oddness. Thus, since $\varepsilon\,ie^{im}\neq0$, the condition~\eqref{eq:perturbation-integral} reduces to
\[
 \int_{-m}^{m}(1-q\cos x)\cos x\,dx
 =2\sin m-q(m+\sin m\cos m)=0,
\]
which holds for our chosen value of $q$. Hence
\[
 c(\varphi)=k\int_0^{2m} v_\theta\,d\theta
 =k(\cos 2m-1,\sin 2m);
\]
since $k=\frac12\cot^2m$, this gives
\[
 c(\varphi)+\frac12u_\varphi=\left(-\frac12,T\right).
\]

We define a $2\pi$-periodic function $h$ by
\[
 h(\theta)=
 \begin{cases}
 c(\theta)\cdot u_\theta+\frac12
      &0\le \theta\le \varphi,\\[2mm]
 a_\varphi\cdot u_\theta
      &\varphi\le \theta\le \pi,\\[2mm]
 -c(\theta-\pi)\cdot u_{\theta-\pi}+\frac12
      &\pi\le \theta\le \pi+\varphi,\\[2mm]
 a_0\cdot u_\theta
      &\pi+\varphi\le \theta\le 2\pi,
 \end{cases}
\]
where
\[
 a_0=\left(\frac12,0\right),\qquad a_\varphi=\left(-\frac12,T\right).
\]
Lemma~\ref{lem:convex} below shows that $h$ is the support function of a convex body; we call that body $R$ in the sequel. A direct check shows the four pieces agree at the breakpoints $\theta=\varphi,\pi,\pi+\varphi,2\pi$, so $h$ is continuous and $2\pi$-periodic: at $\varphi$ this is the relation $a_\varphi=c(\varphi)+\tfrac12u_\varphi$ above; at $\pi$ and $2\pi$ both adjacent pieces equal $\tfrac12$; and at $\pi+\varphi$ both equal $-\tfrac12\cos 2m$, using $c(\varphi)\cdot u_\varphi=\cos^2 m$ (valid since $k=\tfrac12\cot^2m$).

It is useful to name the parts of the boundary. For $0\le \theta\le \varphi$, let
\[
 a(\theta)=c(\theta)+\frac12u_\theta,
 \qquad b(\theta)=c(\theta)-\frac12u_\theta.
\]
Then
\[
 a_0=a(0),\quad b_0=b(0)=\left(-\frac12,0\right),\quad
 a_\varphi=a(\varphi),\quad b_\varphi=b(\varphi)=a_\varphi-u_\varphi.
\]
We call $[0,\varphi]$ and $[\pi,\pi+\varphi]$ the \emph{active} ranges of directions and call the complementary ranges $(\varphi,\pi)$ and $(\pi+\varphi,2\pi)$ the \emph{gaps}: in an active direction the boundary point lies on a curved arc and the width equals the thickness (Lemma~\ref{lem:width}), whereas in a gap $h$ is linear, the supporting line touches at a single corner, and the width exceeds the thickness. The boundary $\partial R$ consists, in positive order, of the \emph{upper active arc} $\Gamma_a=\{a(\theta):0\le \theta\le\varphi\}$, the segment $[a_\varphi,b_0]$, the \emph{lower active arc}
$\Gamma_b=\{b(\theta):0\le \theta\le\varphi\}$, and the segment $[b_\varphi,a_0]$. Figure~\ref{fig:body} shows $R$ together with several thickness chords.

\begin{figure}[t]
\centering
\input{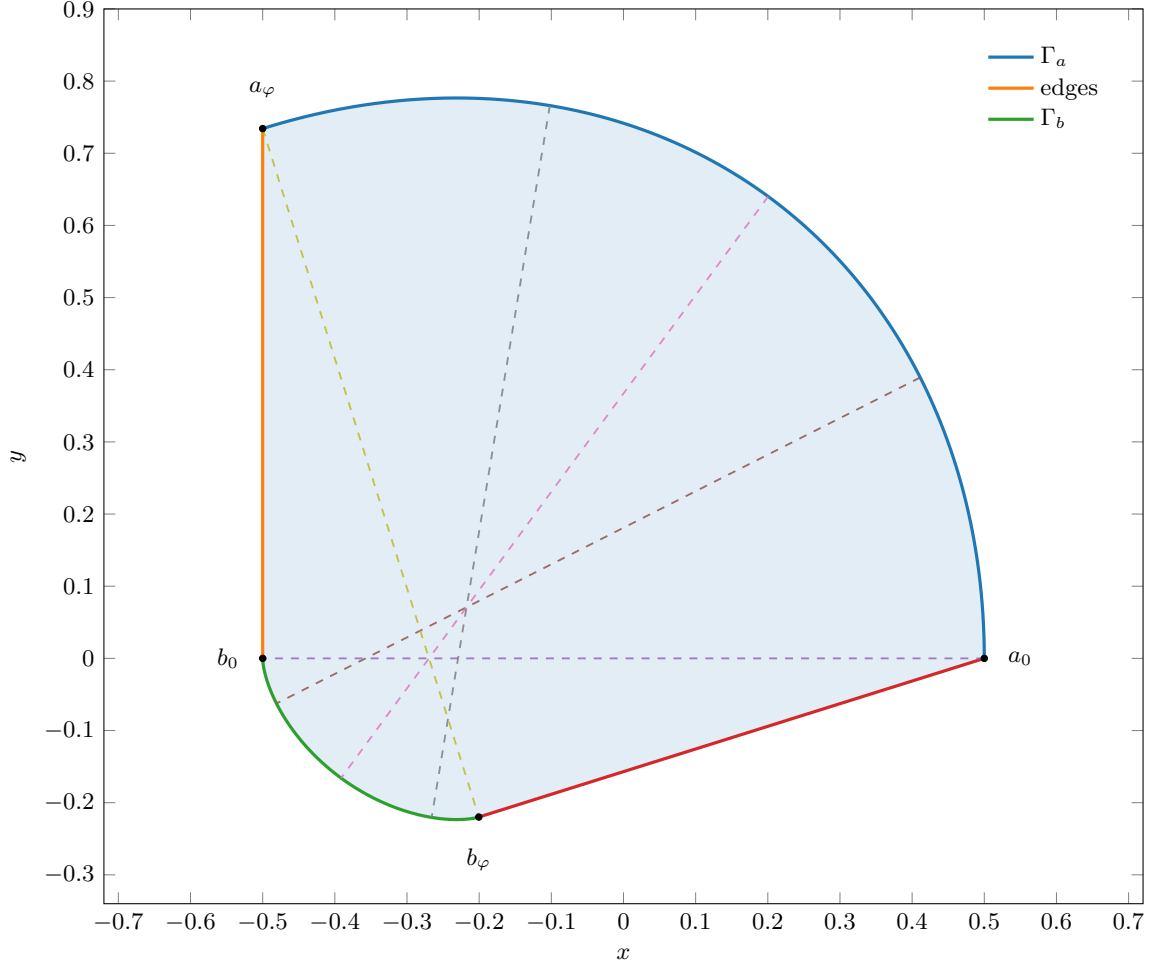}
\caption{The body $R$. The dashed segments are the thickness chords $[a(\theta),b(\theta)]$ (each of length $1$) for five equally-spaced values of $\theta\in[0,\varphi]$.}
\label{fig:body}
\end{figure}

\section{Convexity and thickness}
\begin{lemma}\label{lem:convex}
The function $h$ is the support function of a convex body $R$. Moreover, the two active arcs are strictly convex.
\end{lemma}

\begin{proof}
A continuous $2\pi$-periodic function $h$ is the support function of a convex body if and only if the distribution $\mu=h+h''$ is a non-negative measure \cite{Schneider2013}. Where $h$ is twice differentiable this measure has a \emph{density} $h+h''$ with respect to $d\theta$, equal to the radius of curvature of the boundary there. Where $h'$ jumps it has an \emph{atom} --- a point mass whose weight is the jump $h'(\theta+)-h'(\theta-)$ --- marking a straight boundary edge of that length. We compute $\mu$ on the four smooth pieces and at the two jumps.

On $0<\theta<\varphi$ we have $h=c\cdot u+\tfrac12$. As $c'=rv$,
\[
 h'=c\cdot v,
 \qquad h''=r-c\cdot u,
 \qquad h+h''=\frac12+r(\theta).
\]
On $\pi<\theta<\pi+\varphi$, with $t=\theta-\pi$, the same calculation gives
\[
 h+h''=\frac12-r(t).
\]
On the two gaps $h$ has the form $x\cdot u_\theta$ for a fixed point $x$, and hence $h+h''=0$ there.

It remains to check the sign of the two active densities and the atoms. Write
\[
 \alpha=k+\varepsilon,
 \qquad \beta=\varepsilon q,
 \qquad r(\theta)=\alpha-\beta\cos(\theta-m).
\]
Because $\theta-m\in[-m,m]$ and $0<m<\pi/2$, the range of $r$ is
\begin{gather*}
 r_{\min}=\alpha-\beta=0.1786983693244059\ldots,\\
 r_{\max}=\alpha-\beta\cos m=0.4810926519390665\ldots.
\end{gather*}
Thus we have
\[
 \frac12+r(\theta)\in[0.6786983693244059\ldots,
                       0.9810926519390665\ldots]
\]
and
\[
 \frac12-r(\theta)\in[0.01890734806093346\ldots,
                       0.3213016306755941\ldots],
\]
so both densities are strictly positive.\footnote{We certify the inequalities $0<r<\tfrac12$ by exact rational bounds in Appendix~\ref{app:cert}.}

It remains to locate the atoms. At $\theta=\pi$,
\[
 h'(\pi+)-h'(\pi-)=0-(-T)=T>0,
\]
and at $\theta=\pi+\varphi$,
\[
 h'((\pi+\varphi)+)-h'((\pi+\varphi)-)
 =\left(\frac12+k\right)\sin\varphi=T>0.
\]
At the remaining two breakpoints $h'$ is continuous: $h'(\varphi-)=c(\varphi)\cdot v_\varphi=a_\varphi\cdot v_\varphi=h'(\varphi+)$, and $h'(2\pi-)=a_0\cdot v_0=0=c(0)\cdot v_0=h'(0+)$. Hence the only atoms are the two displayed above, at the normals of the straight edges, and both are positive. Therefore $\mu=h+h''$ is a non-negative measure, with strictly positive density on the two active arcs and positive atoms at the two edges, so $h$ is a legitimate support function; the positive density on the arcs shows that $R$ is genuinely two-dimensional. The densities $\frac12+r$ and $\frac12-r$ are the curvature radii of the two active arcs, so both arcs are strictly convex.
\end{proof}

\begin{remark}
The measure computed in Lemma~\ref{lem:convex} reproduces the oriented boundary described at the end of Section~\ref{sec:construction}. At a direction~$\theta$ where $h$ is differentiable, the boundary point of $R$ with outer normal $u_\theta$ is $x(\theta)=h(\theta)u_\theta+h'(\theta)v_\theta$. On an active range $h=c\cdot u_\theta+\tfrac12$ and $h'=c\cdot v_\theta$, so
\[
 x(\theta)=(c\cdot u_\theta)u_\theta+(c\cdot v_\theta)v_\theta+\tfrac12u_\theta=c(\theta)+\tfrac12u_\theta=a(\theta),
\]
and likewise $b(\theta-\pi)$ on $[\pi,\pi+\varphi]$; these trace the two active arcs. On a gap $h=x_0\cdot u_\theta$ for a fixed point $x_0$ (namely $a_\varphi$ on $(\varphi,\pi)$ and $a_0$ on $(\pi+\varphi,2\pi)$), so $x(\theta)\equiv x_0$: the vanishing density collapses the entire normal interval to that one corner, of which the interval is the normal cone. Each atom is a boundary segment perpendicular to its normal of length equal to the jump $T$; the atom at $\theta=\pi$ is the segment $[a_\varphi,b_0]$, and the atom at $\theta=\pi+\varphi$ is $[b_\varphi,a_0]$. Finally, as $\theta$ increases over $[0,2\pi]$ the normal $u_\theta$ turns once counterclockwise and $x(\theta)$ traverses $\partial R$ once in the positive sense: on $\Gamma_a$, $a'(\theta)=(\tfrac12+r)v_\theta$ is a positive multiple of the counterclockwise tangent $v_\theta$, while on $\Gamma_b$ the outer normal is $u_{\theta+\pi}$, so the counterclockwise tangent is $v_{\theta+\pi}=-v_\theta$ and $b'(\theta)=(r-\tfrac12)v_\theta=(\tfrac12-r)v_{\theta+\pi}$ is again a positive multiple of it, precisely because $r<\tfrac12$. This is the orientation used in the area integral~\eqref{eq:area_basic} and in the containment $\partial R\subseteq K$ demonstrated in the proof of Theorem~\ref{thm:reduced}.
\end{remark}

\begin{lemma}\label{lem:width}
The thickness of $R$ is $\Delta(R)=1$. The width equals $1$ exactly on the active ranges
$[0,\varphi]\cup[\pi,\pi+\varphi]$.
\end{lemma}

\begin{proof}
Let $w(\theta)=h(\theta)+h(\theta+\pi)$. For $0\le \theta\le \varphi$,
\[
 w(\theta)=\left(c(\theta)\cdot u_\theta+\frac12\right)
 +\left(-c(\theta)\cdot u_\theta+\frac12\right)=1.
\]
By $\pi$-periodicity of the width, the same holds on $[\pi,\pi+\varphi]$.

Now let $\varphi\le \theta\le \pi$ and put $z=\pi-\theta\in[0,\psi]$. Then
\[
 w(\theta)=a_\varphi\cdot u_\theta+a_0\cdot u_{\theta+\pi}
 =-\cos\theta+T\sin\theta
 =\cos z+\tan\frac{\psi}{2}\sin z.
\]
Using $1-\cos z=\sin z\tan(z/2)$, we obtain
\[
 w(\theta)-1=\sin z\left(\tan\frac{\psi}{2}-\tan\frac z2\right)\ge 0,
\]
because $0\le z\le\psi<\pi$. Equality occurs only for $z=0$ or $z=\psi$, i.e., for
$\theta=\pi$ or $\theta=\varphi$. The antipodal gap is identical. Hence, the minimum width --- i.e., the thickness --- is exactly $1$. Figure~\ref{fig:width} plots the resulting width function.
\end{proof}

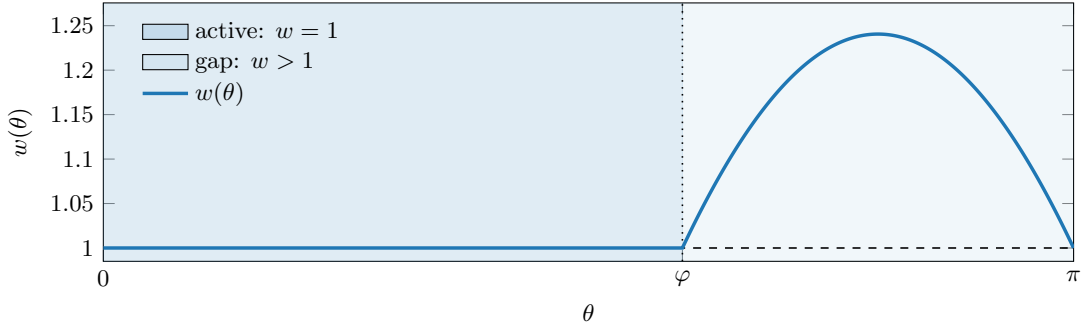
\begin{figure}[t]
\centering
\begin{tikzpicture}
\definecolor{cBlue}{RGB}{31,119,180}%
\begin{axis}[
  width=0.86\linewidth, height=5cm,
  xlabel={$\theta$}, ylabel={$w(\theta)$},
  xmin=0, xmax=3.14159, ymin=0.985, ymax=1.2756,
  xtick={0,1.87500,3.14159}, xticklabels={$0$,$\varphi$,$\pi$},
  ytick={1.00,1.05,1.10,1.15,1.20,1.25},
  legend pos=north west,
  legend cell align=left,
  legend style={draw=none,fill=none,font=\small},
  tick label style={font=\small}, label style={font=\small},
]
\fill[cBlue,fill opacity=0.14] (axis cs:0,0.985) rectangle (axis cs:1.87500,1.2756);
\fill[cBlue,fill opacity=0.06] (axis cs:1.87500,0.985) rectangle (axis cs:3.14159,1.2756);
\addlegendimage{area legend,fill=cBlue,fill opacity=0.14}
\addlegendentry{active: $w=1$}
\addlegendimage{area legend,fill=cBlue,fill opacity=0.06}
\addlegendentry{gap: $w>1$}
\draw[dashed,black,line width=0.6pt] (axis cs:0,1) -- (axis cs:3.14159,1);
\draw[dotted,black,line width=0.6pt] (axis cs:1.87500,0.985) -- (axis cs:1.87500,1.2756);
\addplot[cBlue,line width=1.3pt] coordinates {(0.00000,1.00000) (1.87500,1.00000) (1.87500,1.00000) (1.88136,1.00465) (1.88773,1.00926) (1.89409,1.01384) (1.90046,1.01837) (1.90682,1.02285) (1.91319,1.02730) (1.91955,1.03171) (1.92592,1.03607) (1.93228,1.04039) (1.93865,1.04467) (1.94501,1.04891) (1.95138,1.05310) (1.95774,1.05726) (1.96411,1.06137) (1.97047,1.06543) (1.97684,1.06946) (1.98320,1.07344) (1.98957,1.07737) (1.99593,1.08127) (2.00230,1.08511) (2.00866,1.08892) (2.01503,1.09268) (2.02139,1.09640) (2.02775,1.10007) (2.03412,1.10370) (2.04048,1.10728) (2.04685,1.11082) (2.05321,1.11431) (2.05958,1.11776) (2.06594,1.12116) (2.07231,1.12452) (2.07867,1.12783) (2.08504,1.13110) (2.09140,1.13432) (2.09777,1.13749) (2.10413,1.14062) (2.11050,1.14370) (2.11686,1.14674) (2.12323,1.14973) (2.12959,1.15267) (2.13596,1.15556) (2.14232,1.15841) (2.14869,1.16122) (2.15505,1.16397) (2.16142,1.16668) (2.16778,1.16934) (2.17414,1.17195) (2.18051,1.17452) (2.18687,1.17704) (2.19324,1.17951) (2.19960,1.18193) (2.20597,1.18431) (2.21233,1.18663) (2.21870,1.18891) (2.22506,1.19114) (2.23143,1.19332) (2.23779,1.19546) (2.24416,1.19754) (2.25052,1.19958) (2.25689,1.20157) (2.26325,1.20351) (2.26962,1.20540) (2.27598,1.20724) (2.28235,1.20904) (2.28871,1.21078) (2.29508,1.21248) (2.30144,1.21412) (2.30781,1.21572) (2.31417,1.21727) (2.32054,1.21877) (2.32690,1.22022) (2.33326,1.22162) (2.33963,1.22297) (2.34599,1.22427) (2.35236,1.22552) (2.35872,1.22672) (2.36509,1.22787) (2.37145,1.22897) (2.37782,1.23002) (2.38418,1.23103) (2.39055,1.23198) (2.39691,1.23288) (2.40328,1.23374) (2.40964,1.23454) (2.41601,1.23529) (2.42237,1.23599) (2.42874,1.23665) (2.43510,1.23725) (2.44147,1.23780) (2.44783,1.23830) (2.45420,1.23875) (2.46056,1.23916) (2.46693,1.23951) (2.47329,1.23981) (2.47965,1.24006) (2.48602,1.24026) (2.49238,1.24041) (2.49875,1.24051) (2.50511,1.24056) (2.51148,1.24056) (2.51784,1.24051) (2.52421,1.24041) (2.53057,1.24026) (2.53694,1.24006) (2.54330,1.23981) (2.54967,1.23951) (2.55603,1.23916) (2.56240,1.23875) (2.56876,1.23830) (2.57513,1.23780) (2.58149,1.23725) (2.58786,1.23665) (2.59422,1.23599) (2.60059,1.23529) (2.60695,1.23454) (2.61332,1.23374) (2.61968,1.23288) (2.62604,1.23198) (2.63241,1.23103) (2.63877,1.23002) (2.64514,1.22897) (2.65150,1.22787) (2.65787,1.22672) (2.66423,1.22552) (2.67060,1.22427) (2.67696,1.22297) (2.68333,1.22162) (2.68969,1.22022) (2.69606,1.21877) (2.70242,1.21727) (2.70879,1.21572) (2.71515,1.21412) (2.72152,1.21248) (2.72788,1.21078) (2.73425,1.20904) (2.74061,1.20724) (2.74698,1.20540) (2.75334,1.20351) (2.75971,1.20157) (2.76607,1.19958) (2.77243,1.19754) (2.77880,1.19546) (2.78516,1.19332) (2.79153,1.19114) (2.79789,1.18891) (2.80426,1.18663) (2.81062,1.18431) (2.81699,1.18193) (2.82335,1.17951) (2.82972,1.17704) (2.83608,1.17452) (2.84245,1.17195) (2.84881,1.16934) (2.85518,1.16668) (2.86154,1.16397) (2.86791,1.16122) (2.87427,1.15841) (2.88064,1.15556) (2.88700,1.15267) (2.89337,1.14973) (2.89973,1.14674) (2.90610,1.14370) (2.91246,1.14062) (2.91883,1.13749) (2.92519,1.13432) (2.93155,1.13110) (2.93792,1.12783) (2.94428,1.12452) (2.95065,1.12116) (2.95701,1.11776) (2.96338,1.11431) (2.96974,1.11082) (2.97611,1.10728) (2.98247,1.10370) (2.98884,1.10007) (2.99520,1.09640) (3.00157,1.09268) (3.00793,1.08892) (3.01430,1.08511) (3.02066,1.08127) (3.02703,1.07737) (3.03339,1.07344) (3.03976,1.06946) (3.04612,1.06543) (3.05249,1.06137) (3.05885,1.05726) (3.06522,1.05310) (3.07158,1.04891) (3.07794,1.04467) (3.08431,1.04039) (3.09067,1.03607) (3.09704,1.03171) (3.10340,1.02730) (3.10977,1.02285) (3.11613,1.01837) (3.12250,1.01384) (3.12886,1.00926) (3.13523,1.00465) (3.14159,1.00000)};
\addlegendentry{$w(\theta)$}
\end{axis}
\end{tikzpicture}
\caption{The width function on $[0,\pi]$: the width is identically $1$ on the active range $[0,\varphi]$ and strictly larger than $1$ on the gap $(\varphi,\pi)$.}
\label{fig:width}
\end{figure}

\section{Area}\label{sec:area}
We compute the area from the boundary line integral
\[
 \area(R)=\frac12\oint_{\partial R} x\times dx,
 \qquad x\times y=x_1y_2-x_2y_1.
\]
Let
\[
 p(\theta)=c(\theta)\cdot u_\theta
 =\int_0^\theta r(s)\sin(\theta-s)\,ds.
\]
Since $a'= (r+\tfrac12)v$ and $b'=(r-\tfrac12)v$, while $c\times v=c\cdot u=p$, we have
\[
 a(\theta)\times a'(\theta)+b(\theta)\times b'(\theta)
 =2r(\theta)p(\theta)+\frac12.
\]
The two straight edges contribute
\[
 a_\varphi\times b_0+b_\varphi\times a_0=\frac12\sin\varphi
\]
to the integral $\oint x\times dx$. Therefore
\begin{equation}\label{eq:area_basic}
 \area(R)=\int_0^\varphi r(\theta)p(\theta)\,d\theta
 +\frac{\varphi}{4}+\frac14\sin\varphi
 =\int_0^{2m} r(\theta)p(\theta)\,d\theta
 +\frac m2+\frac14\sin 2m.
\end{equation}

For the closed form, set again $\alpha=k+\varepsilon$ and $\beta=\varepsilon q$, so that
$r(\theta)=\alpha-\beta\cos(\theta-m)$. The inner integral is
\begin{equation}\label{eq:p_explicit}
 p(\theta)=\alpha(1-\cos\theta)
 -\frac{\beta}{2}\bigl(\theta\sin(\theta-m)+\sin m\sin\theta\bigr).
\end{equation}
Indeed, this is obtained by integrating
$\int_0^\theta\cos(s-m)\sin(\theta-s)\,ds$ after the substitution $x=s-m$.
Set $D(\theta)=\theta\sin(\theta-m)+\sin m\sin\theta$, so that $p=\alpha(1-\cos\theta)-\tfrac{\beta}{2}D$. Expanding $r(\theta)p(\theta)=\bigl(\alpha-\beta\cos(\theta-m)\bigr)p$ and grouping by powers of $\alpha,\beta$,
\[
 r(\theta)p(\theta)=\alpha^2(1-\cos\theta)
 -\alpha\beta\Bigl(\tfrac12 D+\cos(\theta-m)(1-\cos\theta)\Bigr)
 +\tfrac12\beta^2\cos(\theta-m)\,D.
\]
Integrating each coefficient over $[0,2m]$ (each an elementary trigonometric integral) gives 
\[
 \int_0^{2m} r(\theta)p(\theta)\,d\theta=\alpha^2 I_{\alpha^2}+\alpha\beta\,I_{\alpha\beta}+\beta^2 I_{\beta^2},
\]
with
\begin{gather*}
 I_{\alpha^2}=2m-\sin 2m,\qquad
 I_{\alpha\beta}=2m\cos m-\tfrac72\sin m+\tfrac12\sin 3m,\\
 I_{\beta^2}=-\tfrac12 m\cos 2m+\tfrac14 m+\tfrac14\sin 2m-\tfrac1{16}\sin 4m.
\end{gather*}
Here $I_{\alpha\beta}$ carries the minus sign of the $\alpha\beta$-term in the grouping (so it is negative). Adding the remaining term $\tfrac m2+\tfrac14\sin 2m$ of \eqref{eq:area_basic} gives
\begin{align*}
\area(R)=&\;2\alpha^2m-\alpha^2\sin 2m
 +2\alpha\beta m\cos m-\frac72\alpha\beta\sin m
 +\frac12\alpha\beta\sin 3m \\
&-\frac12\beta^2m\cos 2m+\frac14\beta^2m
 +\frac14\beta^2\sin 2m-\frac1{16}\beta^2\sin 4m
 +\frac m2+\frac14\sin 2m;
\end{align*}
as $m=15/16$, $\varepsilon=13/20$, $k=\frac12\cot^2m$, and
$q=2\sin m/(m+\sin m\cos m)$, this yields
\[
 \area(R)=0.786215602719126122986809092744\ldots,
\]
whereas
\[
 \frac{\pi}{4}=0.785398163397448309615660845819\ldots.
\]
Thus
\[
 \area(R)-\frac{\pi}{4}
 =0.000817439321677813371148246924\ldots>0,
\]
and
\[
 \frac{\area(R)}{\pi/4}=1.0010407960697816793\ldots.
\]
The single integral in \eqref{eq:area_basic}, with $p$ given by \eqref{eq:p_explicit}, also provides an immediate independent quadrature check. Both sides of $\area(R)>\pi/4$ are explicit constants.\footnote{In Appendix~\ref{app:cert}, we certify the strict inequality by exact rational bounds: $\area(R)>0.78621>0.78540>\pi/4$.}

\section{Reducedness}
\begin{theorem}\label{thm:reduced}
Every convex body $K$ with $K\subseteq R$ and $\Delta(K)\ge 1$ equals $R$; hence, $R$ is reduced.
\end{theorem}

\begin{proof}
We fix a convex body $K$ with $K\subseteq R$ and $\Delta(K)\ge 1$, and let $h_K$ be its support function. Since $K\subseteq R$, we have $h_K\le h$ in every direction. Put $w_K(\theta)=h_K(\theta)+h_K(\theta+\pi)$.

\emph{(i)} For $0\le \theta\le\varphi$, Lemma~\ref{lem:width} gives $w_R(\theta)=1$. Hence
\[
 1\le w_K(\theta)=h_K(\theta)+h_K(\theta+\pi)
 \le h(\theta)+h(\theta+\pi)=1.
\]
Therefore $w_K(\theta)=1$. Since both summands are bounded above by the corresponding summands for $R$, equality of the sums forces
\[
 h_K(\theta)=h(\theta),
 \qquad h_K(\theta+\pi)=h(\theta+\pi)
\]
for every $\theta\in[0,\varphi]$.

\emph{(ii)} Fix $\theta\in(0,\varphi)$. On this open active interval $h$ is $C^2$, and Lemma~\ref{lem:convex} gives positive curvature density and no atom at $\theta$; hence $R$ has a unique boundary point with outer normal $u_\theta$, its support face being the singleton
\[
 F_R(u_\theta)=\{h(\theta)u_\theta+h'(\theta)v_\theta\}=\{a(\theta)\}.
\]
Equivalently, the supporting line $\{x:x\cdot u_\theta=h(\theta)\}$ meets $R$ only at $a(\theta)$. Because $K$ is compact and $h_K(\theta)=h(\theta)$, the maximum of $x\cdot u_\theta$ over $K$ is attained and equals $h(\theta)$, so $K$ meets this line as well; as $K\subseteq R$, the contact point can only be $a(\theta)$, whence we have $a(\theta)\in K$. The opposite supporting line, whose unique contact point on $R$ is $b(\theta)$, gives $b(\theta)\in K$ in the same way.

\emph{(iii)} Hence $a(\theta),b(\theta)\in K$ for all $\theta\in(0,\varphi)$. As $K$ is closed, the four endpoints
$a_0,a_\varphi,b_0,b_\varphi$ also lie in~$K$. By convexity,
\[
 [a_\varphi,b_0]\subseteq K,
 \qquad [b_\varphi,a_0]\subseteq K.
\]
Thus, we see that both arcs and both boundary segments of $R$ are contained in $K$.

\emph{(iv)} We have shown that $\partial R\subseteq K$. A convex set containing $\partial R$ contains
$\conv(\partial R)=R$, so we must have $R\subseteq K$. As $K\subseteq R$ by hypothesis, we conclude that $K=R$.
\end{proof}

\section{Relation to the structure theorem and the quarter-disk}\label{sec:structure}
The structure theory of planar reduced bodies says that such a body is the convex hull of the endpoints of its thickness chords, every extreme point being such an endpoint; its boundary is made up of arcs of curves of constant width $\Delta$, occurring in opposite pairs joined by the thickness chords, together with straight segments, and a maximal boundary segment lies opposite a single point on the parallel supporting line (see \cite{Lassak1990,LassakMartini2011} for precise statements). We record that $R$ conforms to this description, which provides both an independent consistency check on Theorem~\ref{thm:reduced} and some degree of an explanation of the construction.

By Lemma~\ref{lem:convex}, the curvature radii of the two active arcs at the two ends of the chord $[a(\theta),b(\theta)]$, that is, in the opposite normal directions $\theta$ and $\theta+\pi$, are
\[
 \rho_a(\theta)=h(\theta)+h''(\theta)=\tfrac12+r(\theta),
 \qquad
 \rho_b(\theta)=h(\theta+\pi)+h''(\theta+\pi)=\tfrac12-r(\theta),
\]
so that
\begin{equation}
 \rho_a(\theta)+\rho_b(\theta)=1=\Delta,
 \qquad 0<\theta<\varphi.
\label{eq:direction-condition}
\end{equation}
For a curve of constant width the radii of curvature at antipodal points sum to the width \cite{Schneider2013}; the identity $\rho_a+\rho_b=1$ is the condition~\eqref{eq:direction-condition} in infinitesimal form, while $h(\theta)+h(\theta+\pi)=1$ on $[0,\varphi]$ (Lemma~\ref{lem:width}) is its global form, so $\Gamma_a$ and $\Gamma_b$ are opposite arcs of a curve of constant width $\Delta$. The two straight edges are each opposite a single point --- the corners $a_0$ and $a_\varphi$, respectively --- and $R=\conv(\Gamma_a\cup\Gamma_b)$. Consistent with the theory, $R$ is not strictly convex: a strictly convex reduced plane body is of constant width \cite{Dekster1986}, and $R$ is not of constant width, so its boundary must contain a segment --- indeed it carries two segments (and two corners).

It is instructive to compare $R$ with the quarter-disk, the conjectured non-circular extremizer. The quarter-disk of radius $\Delta$ is the convex hull of a single circular arc of radius $\Delta$ together with the center of that arc; its single active arc (a quarter circle, with $\rho=\Delta$) is opposite the single center point (with $\rho=0$, understood in the limiting support-measure sense), a \emph{degenerate} constant-width pair $\rho+\rho_{\mathrm{opp}}=\Delta+0=\Delta$, the active arc spanning exactly a quarter turn and the two radii serving as the segments. (The quarter-disk is also the equality case of Lassak's sharp diameter bound $\diam(R)\le\sqrt2\,\Delta(R)$ \cite{Lassak1990}.) The body $R$ is the \emph{generic} member of the same family: the curvature is redistributed off the degenerate center onto a genuine second arc, so that both opposite radii lie strictly between $0$ and $\Delta$, and the active arc runs past a quarter turn, $\varphi=2m\approx 1.875>\pi/2$. By Theorem~\ref{thm:reduced} this deformation preserves reducedness while enclosing slightly more area. Informally, then, $R$ is a \emph{blunted quarter-disk}: the degenerate apex of the quarter-disk --- its center point, where all the opposite curvature is concentrated --- is blunted into a genuine arc, at fixed thickness $1$.

We stress that reducedness is a rigidity property of the supporting lines and carries no implication of maximal area. Among reduced bodies of thickness $1$, the disk and the quarter-disk have area $\pi/4$, whereas every reduced polygon has area strictly below $\pi/4$. The strict inequality $\area(R)>\pi/4$ is therefore a separate, quantitative phenomenon, established by the explicit area computation of Section~\ref{sec:area} and not by the rigidity argument that gives reducedness.

\section{Concluding remarks}
The body $R$ settles the conjecture \eqref{eq:conj} in the negative: it is reduced, has thickness $1$, and encloses area $0.786215\ldots>\pi/4$, an absolute excess of about $8.17\times10^{-4}$ (a relative excess of about $0.1\%$). More importantly, it reopens the underlying extremal problem. Write
\[
 M=\sup\{\area(C): C\text{ reduced},\ \Delta(C)=1\}.
\]
The disk and the quarter-disk give $M\ge\pi/4$, and that was conjectured to be an equality; our example shows that $M>\pi/4$ strictly. The exact value of $M$, whether it is attained, and if so by what body, are open~questions.

We have not explicitly attempted to make the excess large. The parameters $m=15/16$ and $\varepsilon=13/20$ were chosen so that the construction admits a clean closed form, not so as to maximize area. In fact the construction closes for every $m\in(0,\pi/2)$ and every $\varepsilon$ --- the choice $q=q(m)$ together with $k=\tfrac12\cot^2m$ forces $c(\varphi)+\tfrac12u_\varphi=(-\tfrac12,T)$ irrespective of $\varepsilon$ --- while the properties we rely on are open conditions: positivity of the curvature densities $\tfrac12\pm r$ makes $R$ convex, renders each active support face a singleton, and orients the lower arc, so that the reducedness proof of Theorem~\ref{thm:reduced} applies verbatim, together with the strict area excess $\area>\pi/4$. Both hold with room at $(\tfrac{15}{16},\tfrac{13}{20})$, where in fact $0<r<\tfrac12$ (see Appendix~\ref{app:cert}). These conditions therefore persist on a neighborhood of $(\tfrac{15}{16},\tfrac{13}{20})$, giving a whole family of reduced bodies with area above $\pi/4$, whose supremum we have not determined. Thus $\area(R)$ is only a lower bound for $M$, and a numerical or analytic optimization --- over this family and beyond --- would sharpen it and perhaps point to a candidate maximizer.

Several other questions suggest themselves as well. Are the disk and the quarter-disk even \emph{local} maxima of area among reduced bodies, or --- as our perturbation off the quarter-disk suggests --- is neither locally extremal?  Is there an analogous curved perturbation off the disk whose area would also exceed $\pi/4$?  And while the spherical analogue for reduced polygons has been studied by Liu--Chang--Su \cite{LiuChangSu2020}, and related reducedness problems in the hyperbolic plane by S\'agmeister \cite{Sagmeister2024}, the Euclidean bound fails here through an essentially curved mechanism; it is unclear whether that mechanism transfers to the spherical or hyperbolic settings for general reduced bodies, or whether the curved geometry there protects the bound.

\appendix
\section{An exact certificate for the numerical inequalities}\label{app:cert}

Two steps rely on numerical inequalities: the bounds $0<r(\theta)<\tfrac12$, which give the strict positivity of the densities $\tfrac12\pm r(\theta)$ in Lemma~\ref{lem:convex}, and the strict area excess $\area(R)>\pi/4$ of Section~\ref{sec:area}. We certify both exactly without floating-point evaluation. The only transcendental inputs are $S=\sin\tfrac{15}{16}$, $C=\cos\tfrac{15}{16}$, and $\pi$; once these are trapped between explicit rationals, every other quantity is a rational function of $S$ and $C$ with rational coefficients, and each inequality becomes a comparison of rationals.

\smallskip
\noindent\emph{Rational brackets for $S$ and $C$.}  For $x=\tfrac{15}{16}<1$ the Maclaurin series
\[
 \sin x=\sum_{n\ge0}\frac{(-1)^n x^{2n+1}}{(2n+1)!},
 \qquad
 \cos x=\sum_{n\ge0}\frac{(-1)^n x^{2n}}{(2n)!}
\]
are alternating with strictly decreasing terms, so consecutive partial sums bracket the value and the truncation error is at most the first omitted term. Summing through $n=6$ (omitted terms below $10^{-11}$) gives
\[
 S\in[\,0.8060811,\ 0.8060812\,],\qquad
 C\in[\,0.5918050,\ 0.5918051\,],
\]
with each endpoint a rational with denominator $10^{7}$.

\smallskip
\noindent\emph{Reduction to $S$ and $C$.}  Using $\sin2m=2SC$, $\cos2m=1-2S^2$, $\sin3m=S(3-4S^2)$ and $\sin4m=4SC(1-2S^2)$, the constants
\[
 k=\frac{C^2}{2S^2},\qquad q=\frac{2S}{m+SC},\qquad \alpha=k+\varepsilon,\qquad \beta=\varepsilon q,
\]
and the closed form for $\area(R)$ are rational functions of $S$ and $C$ over $\mathbb Q$ (with $m=\tfrac{15}{16}$ and $\varepsilon=\tfrac{13}{20}$). The denominators $S^2$ and $m+SC$ are bounded away from $0$ by the brackets, so substituting the brackets and evaluating by interval arithmetic over $\mathbb Q$ is valid.

\smallskip
\noindent\emph{Certified enclosures.}  For $\theta\in[0,2m]$ the argument $\theta-m$ lies in $[-m,m]$, so $\cos(\theta-m)\in[\cos m,1]$; since $\beta>0$, the radius $r(\theta)=\alpha-\beta\cos(\theta-m)$ attains its minimum $r_{\min}=\alpha-\beta$ at $\theta=m$ and its maximum $r_{\max}=\alpha-\beta\cos m$ at $\theta=0,2m$, and lies between these throughout. Interval evaluation of the two endpoints $r_{\min},r_{\max}$ gives
\[
 r(\theta)\in[\,0.17869,\ 0.48110\,]\subset\bigl(0,\tfrac12\bigr),
 \qquad\text{so }\ \tfrac12+r>0\ \text{ and }\ \tfrac12-r\ge 0.0189>0,
\]
and the closed form evaluates to
\[
 \area(R)\in[\,0.78621,\ 0.78622\,].
\]
Finally, Machin's formula $\tfrac\pi4=4\arctan\tfrac15-\arctan\tfrac1{239}$, using the alternating arctangent series with an upper bound for $\arctan\tfrac15$ through the $x^9/9$ term and a lower bound for $\arctan\tfrac1{239}$ through the $-x^3/3$ term, gives $\pi/4<0.78540$ (the omitted tails are below $10^{-8}$). Hence we have
\[
 \area(R)\ \ge\ 0.78621\ >\ 0.78540\ >\ \frac{\pi}{4},
\]
the required strict inequality. Every decimal endpoint displayed above is an exact rational (a terminating decimal), and every step is a comparison of explicit rationals; the script \texttt{certificate\_check.py}, provided as an ancillary file, carries out this exact rational interval arithmetic and verifies all of the claimed enclosures.


\begin{thebibliography}{1}

\bibitem{Dekster1986}
Boris~V. Dekster.
\newblock Reduced, strictly convex plane figure is of constant width.
\newblock {\em Journal of Geometry}, 26(1):77--81, 1986.

\bibitem{Heil1978}
E.~Heil.
\newblock {K}leinste konvexe {K}\"{o}rper gegebener {D}icke.
\newblock Preprint No.~453, Fachbereich Mathematik der TH Darmstadt, 1978.

\bibitem{Lassak1990}
Marek Lassak.
\newblock Reduced convex bodies in the plane.
\newblock {\em Israel Journal of Mathematics}, 70(3):365--379, 1990.

\bibitem{Lassak2005}
Marek Lassak.
\newblock Area of reduced polygons.
\newblock {\em Publicationes Mathematicae Debrecen}, 67(3-4):349--354, 2005.

\bibitem{LassakMartini2011}
Marek Lassak and Horst Martini.
\newblock Reduced convex bodies in {E}uclidean space---a survey.
\newblock {\em Expositiones Mathematicae}, 29(2):204--219, 2011.

\bibitem{LiuChangSu2020}
Cen Liu, Yanxun Chang, and Zhanjun Su.
\newblock The area of reduced spherical polygons.
\newblock {\em Rocky Mountain Journal of Mathematics}, 52(2):599--607, 2022.

\bibitem{Pal1921}
Julius P\'{a}l.
\newblock {E}in {M}inimumproblem f\"{u}r {O}vale.
\newblock {\em Mathematische Annalen}, 83(3):311--319, 1921.

\bibitem{Sagmeister2024}
{\'A}d{\'a}m Sagmeister.
\newblock On the perimeter, diameter and circumradius of ordinary hyperbolic
  reduced polygons.
\newblock {\em Canadian Mathematical Bulletin}, 68(3):950--965, 2025.

\bibitem{Schneider2013}
Rolf Schneider.
\newblock {\em Convex Bodies: The {B}runn--{M}inkowski Theory}, volume 151 of
  {\em Encyclopedia of Mathematics and its Applications}.
\newblock Cambridge University Press, Cambridge, 2nd edition, 2013.

\end{thebibliography}
\end{document}